\documentclass{amsart}

\usepackage{amsfonts,amssymb,amsmath,enumerate,verbatim}

\usepackage[all]{xy}
\SelectTips{cm}{}

\newcommand{\bass}[3][R]{\mu^{#2}_{#1}(#3)}
\newcommand{\rI}[2][R]{I_{#1}^{#2}(t)}
\newcommand{\rP}[2][R]{P^{#1}_{#2}(t)}

\newcommand\BZ{{\mathbb Z}}

\newcommand{\cent}[1]{{#1}^{\mathsf c}}
\newcommand\Coker{\operatorname{Coker}}
\newcommand{\col}{\colon}

\newcommand\dd{\partial}

\newcommand{\Tor}[4]{\operatorname{Tor}_{#1}^{#2}(#3,#4)}
\newcommand{\Ext}[4]{\operatorname{Ext}^{#1}_{#2}(#3,#4)}

\newcommand{\Extv}[4]{\widehat{\operatorname{Ext}}{\vphantom E}^{#1}_{#2}(#3,#4)}
\newcommand{\eval}[2]{{\varepsilon^{#1}_{#2}}}

\newcommand\fm{{\mathfrak m}}
\newcommand\fn{{\mathfrak n}}

\newcommand\Hom{\operatorname{Hom}}

\newcommand{\HH}[2]{\operatorname{H}^{#1}(#2)}

\newcommand\Ker{\operatorname{Ker}}

\newcommand{\lra}{\longrightarrow}
\newcommand{\ov}{\overline}
\newcommand{\wh}{\widehat}
\newcommand{\naty}[3]{{\eta^{#1}(#2,#3)}}
\newcommand{\pdim}{\operatorname{pd}}

\newcommand{\rank}{\operatorname{rank}}

\newcommand{\xra}{\xrightarrow}

\newtheorem{theorem}{Theorem}[section]
\newtheorem*{Theorem}{Theorem}
\newtheorem{proposition}[theorem]{Proposition}
\newtheorem{lemma}[theorem]{Lemma}
\newtheorem{corollary}[theorem]{Corollary}

\theoremstyle{definition}

\theoremstyle{remark}
\newtheorem{remark}[theorem]{Remark}

\newtheorem{notes}[theorem]{Notes}

\newtheorem{chunk}[theorem]{}

\numberwithin{equation}{theorem}

\begin{document}

\title[Bass number over local rings]{Bass numbers over local rings\\ via stable cohomology}

\author[L.~L.~Avramov]{Luchezar L.~Avramov}
\address{Department of Mathematics,
University of Nebraska, Lincoln, NE 68588, U.S.A.}
\email{avramov@math.unl.edu}

\author[S.~B.~Iyengar]{Srikanth B.~Iyengar}
\address{Department of Mathematics,
University of Nebraska, Lincoln, NE 68588, U.S.A.}
\email{siyengar2@unl.edu}
\thanks{Research partly supported by NSF grants DMS-1103176 (LLA) and DMS-1201889 (SBI)}

\date{Version from 21st August 2012}


\begin{abstract} 
For any non-zero finite module $M$ of finite projective dimension over a noetherian local ring 
$R$ with maximal ideal $\fm$ and residue field $k$, it is proved that the natural map $\Ext{}RkM\to\Ext{}Rk{M/\fm M}$ is 
non-zero when $R$ is regular and is zero otherwise.   A noteworthy aspect of the proof is the 
use of stable cohomology. Applications include computations of Bass series over certain local rings. 
\end{abstract}

\maketitle

\section*{Introduction}

Let $(R,\fm,k)$ denote a commutative noetherian local ring with maximal ideal $\fm$ and residue 
field $k$; when $R$ is not regular we say that it is \emph{singular}.
 
This article revolves around the following result:

  \begin{Theorem}
If $(R,\fm,k)$ is a singular local ring and $M$ an $R$-module of finite projective dimension, then 
$\Ext{}Rk{\pi^M}=0$ for the canonical map $\pi^M\col M\to M/\fm M$.
  \end{Theorem}

Special cases, known for a long time, are surveyed at the end of Section~2.  Even in those cases 
our proof is new.  It utilizes a result of Martsinkovsky \cite{Mar} through properties of Vogel's stable 
cohomology functors \cite{Go, AV} recalled in Section 1.  It also suggests extensions to DG modules 
over certain commutative DG algebras; see~\cite{AI:eval}.  Applications of the theorem include new 
criteria for regularity of local rings (in Section 2) and explicit computations of Bass numbers of 
modules (in Section 3).

\section{Stable cohomology}
 \label{sec:stable cohomology}
In this section we recall the construction of stable cohomology and basic results required in the sequel. 
The approach we adopt is based on a construction by Vogel, and described in Goichot~\cite{Go}; see also \cite{AV}.

Let $R$ be an associative ring and let $\cent R$ denote its center.  Given left $R$-modules $L$ and $M$, 
choose projective resolutions $P$ and $Q$ of $L$ and $M$, respectively.  Recall that a homomorphism
$P\to Q$ of degree $n$ is a family $\beta=(\beta_i)_{i\in\BZ}$ of $R$-linear maps $\beta_i\col P_i\to Q_{i+n}$;
that is, an element of the $\cent R$-module
\[
{\Hom}_{R}(P,Q)_{n} = \prod_{i\in\BZ}\Hom_{R}(P_{i},Q_{i+n})\,.
  \]
This module is the $n$-th component of a complex ${\Hom}_{R}(P,Q)$, with differential 
  \[
\dd_n(\beta)=\dd^Q_{i+n}\beta_i-(-1)^n\beta_{i-1}\dd^P_n\,.
  \]
The maps $\beta\col P\to Q$ with $\beta_{i}=0$ for $i\gg 0$ form a subcomplex with component
\[
\ov{\Hom}_{R}(P,Q)_{n} = \coprod_{i\in\BZ}\Hom_{R}(P_{i},Q_{i+n})\quad\text{for}\quad n\in\BZ\,.
\]

We write $\wh{\Hom}_{R}(P,Q)$ for the quotient complex.  It is independent of the choices of $P$ and 
$Q$ up to $R$-linear homotopy, and so is the exact sequence of complexes
  \begin{equation}
\label{eq:bes}
0\lra \ov{\Hom}_{R}(P,Q)\lra\Hom_{R}(P,Q) \xra{\ \theta\ } \wh{\Hom}_{R}(P,Q)\lra 0\,.
\end{equation}
The \emph{stable cohomology} of the pair $(L,M)$ is the graded $\cent R$-module 
$\Extv{}RLM$ with
\[
\Extv nRLM = \HH{n}{\wh{\Hom}_{R}(P,Q)}\quad\text{for each}\quad n\in\BZ\,.
\]
It comes equipped with functorial homomorphisms of graded $\cent R$-modules
  \begin{equation}
    \label{eq:diag}
{\Ext{n}RLM} \xra{\,\naty{n} LM\,}{\Extv{n}RLM}
\quad\text{for all}\quad n\in\BZ\,.
  \end{equation}

\begin{chunk}
\label{ch:sext-pdim}
If $\pdim_{R}L$ or $\pdim_{R}M$ is finite, then $\Extv{n}RLM=0$ for all $n\in\BZ$.
\end{chunk}

Indeed, in this case we may choose $P$ or $Q$ to be a bounded complex.  The definitions then
yield $\ov{\Hom}_{R}(P,Q)=\Hom_{R}(P,Q)$, and hence $\wh{\Hom}_{R}(P,Q)=0$.

\begin{chunk}
For a family $\{M_j\}_{j\in J}$ of $R$-modules and every integer $n$ the canonical inclusions
$M_j\to \coprod_{j\in J}M_j$ induce, by functoriality, a commutative diagram of $\cent R$-modules 
   \begin{equation}
    \label{eq:diag}
\xymatrixcolsep{6pc}
\xymatrixrowsep{2pc}
  \begin{gathered}
\xymatrix{
\Ext{n}RL{\coprod_{j\in J}M_j} 
\ar@{->}[r]^{\naty{n} L{\coprod_{j\in J}M_j}}
& \Extv{n}RL{\coprod_{j\in J}M_j}  
  \\
{\coprod_{j\in J}\Ext{n}RL{M_j}} 
\ar@{->}[u] 
\ar@{->}[r]^{{\coprod_{j\in J}\naty{n} L{M_j}}} 
&  \coprod_{j\in J}{\Extv{n}RL{M_j}}
\ar@{->}[u] 
}
  \end{gathered}
  \end{equation}
\end{chunk}

  \begin{proposition}
\label{prop:stable-sums}
Suppose $L$ admits a resolution by finite projective $R$-modules.

For every integer $n$ the vertical maps in \eqref{eq:diag} are bijective.  
In particular, the map $\naty{n} L{\coprod_{j\in J}M_j}$ is injective or surjective 
for some $n$ if and only if $\naty{n} L{M_j}$ has the corresponding property for every $j\in J$.
  \end{proposition}

\begin{proof}
Let $P$ be a resolution of $L$ by finite projective $R$-modules and $Q_j$ a projective resolution 
of $M_j$.  The complex $\coprod_{j\in J}Q_j$ is a projective resolution of $\coprod_{j\in J}M_j$, and
we have a commutative diagram of morphisms of complexes of $\cent R$-modules
\[
\xymatrixcolsep{.8pc}
\xymatrixrowsep{2pc}
\xymatrix{
0 \ar@{->}[r] & \ov{\Hom}_{R}(P,\coprod_{j\in J}Q_j)\ar@{->}[r] 
		&\Hom_{R}(P,\coprod_{j\in J}Q_j) \ar@{->}[r] &\wh{\Hom}_{R}(P,\coprod_{j\in J}Q_j) \ar@{->}[r] & 0
  \\
0 \ar@{->}[r] & \coprod_{j\in J}{\ov{\Hom}_{R}(P,Q_j)} \ar@{->}[r] \ar@{->}[u]^{\ov\varkappa}
		&  \coprod_{j\in J}{\Hom_{R}(P,Q_j)} \ar@{->}[r] \ar@{->}[u]^{\varkappa} 
			&\coprod_{j\in J}{\wh{\Hom}_{R}(P,Q_j)} \ar@{->}[r] \ar@{->}[u]^{\widehat\varkappa} & 0
}
\]
with natural vertical maps. The map $\HH n{\varkappa}$ is bijective, as it represents  
\[
\coprod_{j\in J}{\Ext{n}RL{M_j}}\to \Ext{n}RL{\coprod_{j\in J}M_j}\,,
\]
which is bijective due to the hypothesis on $L$.  As $\ov\varkappa$ is evidently bijective, 
$\HH n{\wh\varkappa}$ is an isomorphism.  The right-hand square of the diagram above
induces \eqref{eq:diag}. 
\end{proof}

\section{Local rings}

The next theorem is the main result of the paper.  It concerns the maps
  \[
\Ext{n}Rk{\beta}\col \Ext{n}Rk{M}\to\Ext{n}Rk{V}
  \]
induced by some homomorphism $\beta\col M\to V$, and is derived from a 
result of Martsinkovsky \cite{Mar} by using properties of stable cohomology, 
recalled above.

\begin{theorem}
\label{thm:lescot}
Let $(R,\fm,k)$ be a local ring and $V$ an $R$-module such that $\fm V=0$.

If $R$ is singular and $\beta\col M\to V$ is an $R$-linear map that factors through 
some module of finite projective dimension, then 
  \[
\Ext{n}Rk{\beta}=0 \quad\text{for all}\quad n\in\BZ\,.
  \]
   \end{theorem}

\begin{proof}
By hypothesis, $\beta$ factors as $M\xra{\gamma}N\xra{\delta}V$ with $N$ of finite projective dimension. The 
following diagram
\[
\xymatrixcolsep{.7pc}
\xymatrixrowsep{2pc}
\xymatrix{
\Ext{n}RkM 
\ar@{->}[rrr]^-{\Ext{n}Rk{\beta}} 
\ar@{->}[drrr]_-{\Ext{n}Rk{\gamma}\quad}
&&&\Ext{n}Rk{V}
\ar@{->}[rrr]^-{\naty{n}k{V}}
&&&\Extv{n}Rk{V}
 \\
&&&\Ext{n}RkN 
\ar@{->}[rrr]^-{\naty{n}kN}
\ar@{->}[u]_-{\Ext{n}Rk{\delta}} 
&&&\Extv{n}Rk{N}
\ar@{->}[u]_-{\Extv{n}Rk{\delta}}
\ar@{=}[r]
&0
}
\]
commutes due to the naturality of the maps involved; the equality comes from~\ref{ch:sext-pdim}.  

The map $\naty{n} kk$ is injective by \cite[Theorem 6]{Mar}.  Proposition \ref{prop:stable-sums} shows
that ${\naty{n}k{V}}$ is injective as well, so the diagram yields $\Ext {n}Rk\beta=0$.
  \end{proof}

Note that no finiteness condition on $M$ is imposed in the theorem.  This 
remark is used in the proof of the following corollary, which deals with the maps
  \[
\Tor{n}Rk{\alpha}\col \Tor{n}Rk{V}\to\Tor{n}Rk{M}
  \]
induced by some homomorphism $\alpha\col V\to M$.

\begin{corollary}
\label{cor:tor}
If $R$ is singular and $\alpha\col V\to M$ is an $R$-linear map that factors 
through some module of finite injective dimension, then 
  \[
\Tor{n}Rk{\alpha}=0\quad\text{for all}\quad n\in\BZ\,.
  \]
\end{corollary}

\begin{proof}
Set $(-)^\vee=\Hom_R(-,E)$, where $E$ is an injective envelope of the $R$-module~$k$.  Let
$V\to L\to M$ be a factorization of $\alpha$ with $L$ of finite injective dimension.  
By Ishikawa \cite[1.5]{Is} the module $L^\vee$ has finite flat dimension, so it has finite 
projective dimension by Jensen \cite[5.8]{Je}.  As $\fm(V^\vee)=0$ and $\alpha^\vee$ factors through 
$L^\vee$, Theorem \ref{thm:lescot} gives $\Ext{n}Rk{\alpha^\vee}=0$.  The natural isomorphism 
$\Ext{n}Rk{-^\vee}\cong\Tor{n}Rk-^\vee$ now yields $\Tor{n}Rk{\alpha}^\vee=0$, whence 
we get $\Tor{n}Rk{\alpha}=0$, as desired.
  \end{proof}

Next we record an elementary observation, where $(-)^*=\Hom_R(-,R)$.

\begin{lemma}
\label{lem:regular}
Let $(R,\fm,k)$ be a local ring and $\chi\col X\to Y$ an $R$-linear map.

If $\Coker(\chi)$ has a non-zero free direct summand, then 
$\Ker(\chi^*)\nsubseteq\fm Y^*$ holds.

When $Y$ is free of finite rank the converse assertion holds as well.
  \end{lemma}

\begin{proof}
The condition on $\Coker(\chi)$ holds if and only if there is an epimorphism $\Coker(\chi)\to R$; that is, an 
$R$-linear map $\upsilon\col Y\to R$ with $\upsilon\chi=0$ and $\upsilon(Y)\nsubseteq\fm$. 

When such a $\upsilon$ exists it is in $\Ker(\chi^*)$, but not in $\fm Y^{*}$, for otherwise $\upsilon(Y)\subseteq\fm$.

When $Y$ is finite free and $\Ker(\chi^*)\nsubseteq\fm Y^*$ holds, pick 
$\upsilon$ in $\Ker(\chi^*)\smallsetminus\fm Y^*$.  Since $Y^*$ is finite free, 
$\upsilon$ can be extended to a basis of $Y^*$, hence $\upsilon(Y)=R$.
   \end{proof}

The theorem in the introduction is the crucial implication in the next result:

\begin{theorem}
\label{thm:regular}
Let $(R,\fm,k)$ be a local ring.  For each $R$-module $M$, let 
\[
\eval{n}M=\Ext{n}Rk{\pi^M}\col \Ext{n}RkM \to \Ext{n}Rk{M/\fm M}
\]
be the map of  $R$-modules induced by the natural map $\pi^M\col M\to M/\fm M$.
 
The following conditions are equivalent.
\begin{enumerate}[{\quad\rm(i)}]
\item
$R$ is regular.
\item
$\eval nR\ne0$ for some integer $n$.
\item
$\eval{n}M\ne0$ for some $R$-module $M$ with $\pdim_{R}M\!<\!\infty$ and some integer $n$.
\item
$\eval dM\ne0$ for every finite $R$-module $M\ne0$ and for $d=\dim R$.
\item
$\Coker(\dd^F_{n})$, where $F$ is a minimal free resolution of $k$ over $R$, has a non-zero free 
direct summand for some integer $n$.
\end{enumerate}
\end{theorem}

\begin{proof}
Set $G=\Hom_{R}(F,R)$ with $F$ as in (v).  From $\Hom_{R}(F,M)\cong G\otimes_RM$ 
and $\dd(G\otimes M)\subseteq\fm(G\otimes M)$ (by the minimality of $F$) we get a commutative diagram
  \begin{equation*}
\xymatrixcolsep{.7pc}
\xymatrixrowsep{2pc}
  \begin{gathered}
\xymatrix{
\Ext{n}RkM
\ar@{->}[rrrr]^-{\eval{n}M}
&&&& \Ext{n}Rk{M/\fm M}
	 \\
\operatorname{H}_{-n}(G\otimes_RM)
\ar@{->}[u]^{\cong} 
\ar@{->}[rrrr]^-{\operatorname{H}_{-n}(G\otimes_R\pi^{M})}
&&&&\operatorname{H}_{-n}(G\otimes_R(M/\fm M))
\ar@{=}[r]
\ar@{->}[u]_{\cong} 
&G_{-n}\otimes(M/\fm M)
}
  \end{gathered}
  \end{equation*}

(i)$\implies$(iv). 
As $R$ is regular we can take $F$ to be the Koszul complex on a minimal generating set of $\fm$.
This gives $G_{d}=R$, an isomorphism $\operatorname{H}_{-d}(G\otimes_{R}\pi^{M})$, and an 
inequality $M/\fm M\ne 0$ by Nakayama's Lemma; now the diagram yields $\eval{d}M\ne0$.

(iv)$\implies$(ii)$\implies$(iii).  These implications are tautologies.

(iii)$\implies$(i).  This implication is a special case of Theorem~\ref{thm:lescot}.

(ii)$\iff$(v).  The preceding diagram shows that the condition $\eval nR\ne0$ is equivalent to 
$\Ker(\dd^{G}_{-n})\nsubseteq\fm G_{-n}$.  Thus, the desired assertion follow from Lemma \ref{lem:regular}.
  \end{proof}

 \begin{notes}
   \label{notes:2}
The equivalence of conditions (i) and (ii) in Theorem \ref{thm:regular} was proved by Ivanov~\cite[Theorem 2]{Iv}
when $R$ is Gorenstein and by Lescot~\cite[1.4]{Le} in general.  

The equivalence of (i) and (v) is due to Dutta~\cite[1.3]{Du}.  As shown above, it follows from Lescot's theorem
via the elementary Lemma  \ref{lem:regular}.  Martsinkovsky deduced Dutta's theorem from 
\cite[Theorem 6]{Mar}, and used the latter to prove regularity criteria different from (ii), (iii), and (iv) in 
Theorem~\ref{thm:regular}; see \cite[p.\,11]{Mar}.  
 \end{notes}

\section{Bass numbers of modules}
  \label{sec:bass modules}

The \emph{$n$th Bass number} of a module $M$ over a local ring $(R,\fm,k)$ is the integer
\[
\bass nM = \rank_{k} \Ext nRkM\,.
\]

In what follows, given a homomorphism $\beta\col M\to N$ and an $R$-submodule 
$N'\subseteq N$ we let $M\cap N'$ denote the submodule $\beta^{-1}(N')$ of $M$.

\begin{theorem}
\label{thm:bass}
Let $(R,\fm,k)$ be a local ring, $M\to N$ an $R$-linear map, and set
  \[
r=\rank_{k}(M/M\cap\fm N)\,.
  \]

If $R$ is singular and $\pdim_{R}N$ is finite, then for each $n\in\BZ$ there is an equality
\[
\bass n{M\cap \fm N} = \bass nM + r \bass{n-1}k\,.
\]
\end{theorem}

\begin{proof}
Set $\ov M= M/(M\cap\fm N)$ and $\ov N = N/\fm N$, and let $\pi\col M\to \ov M$ and 
$\iota\col\ov M\to\ov N$ be the induced maps.   They appear in a commutative diagram
with exact row
\[
\begin{gathered}
\xymatrixcolsep{2.5pc}
\xymatrixrowsep{2pc}
\xymatrix{
&M \ar@{->}[d]_{\pi} \ar@{->}[r] & N \ar@{->}[d] \\
0\ar@{->}[r]
&\ov M \ar@{->}[r]^{\iota} & \ov N}
\end{gathered}
\]
Since $\iota$ is $k$-linear, it is split, so we get a commutative diagram with exact row
\[
\begin{gathered}
\xymatrixcolsep{2pc}
\xymatrixrowsep{2pc}
\xymatrix{
& \Ext{}RkM \ar@{->}[d]_-{\Ext{}Rk{\pi}} \ar@{->}[rr] && \Ext{}RkN \ar@{->}[d]^-{0}
  \\
0 \ar@{->}[r]
& \Ext{}Rk{\ov M} \ar@{->}[rr]^-{\Ext{}Rk{\iota}} && \Ext{}Rk{\ov N} }
\end{gathered}
\]
and zero map due to Theorem~\ref{thm:lescot}.  It implies $\Ext{}Rk{\pi}=0$.

By definition, there exists an exact sequence of $R$-modules
\[
0\lra (M\cap \fm N)\lra M \xra{\ \pi\ } \ov M\lra 0
\]
As $\Ext{}Rk{\pi}=0$, its cohomology sequence yields an exact sequence
\[
0\lra \Ext {n-1}Rk{\ov M}\lra \Ext nRk{M\cap \fm N}\lra \Ext nRkM\lra 0
\]
of $k$-vector spaces for each integer $n$.  Computing ranks over $k$ and using the isomorphism 
$\Ext{}Rk{\ov M}\cong \Ext{}Rkk\otimes_{k}{\ov M}$, we obtain the desired equality.
\end{proof}

Recall that the $n$th \emph{Betti number} of $M$ is the integer $\beta^R_n(M)=\rank_k\Ext nRMk$.

\begin{corollary}
\label{cor:max}
Assume that $R$ is singular and $N\supseteq M\supseteq\fm N$ holds, and set 
\[
s=\rank_k(N/M)\,.
 \]
If $N$ is finite and $\pdim_RN=p<\infty$ holds, then for each  $n\in\BZ$ there is an equality
  \begin{equation}
    \label{eq:sum}
\bass nM = \sum_{i=0}^p\bass{n+i}{R}\beta^R_{i}(N)+s \beta^R_{n-1}(k)\,.
  \end{equation}
  \end{corollary}

\begin{proof}
The hypotheses give $M\cap \fm N=\fm N$ and $r=\rank_k(M/\fm N)$.  By applying Theorem \ref{thm:bass}
to the submodules $M\subseteq N$ and $\fm N\subseteq N$ we obtain
  \begin{align*}
\bass nM
&= \bass n{\fm N} -  r\beta^R_{n-1}(k)\\
&= \bass n{N} + \rank_k(N/\fm N)\beta^R_{n-1}(k)-r\beta^R_{n-1}(k)\\
&= \bass n{N} + s\beta^R_{n-1}(k)\,.
  \end{align*}
As $\pdim_{R}N$ is finite, Foxby \cite[4.3(2)]{Fo} yields $\bass n{N}=\sum_{i=0}^p\bass{n+i}{R}\beta^R_{i}(N)$.
 \end{proof}

\begin{remark}
  \label{rem:regular}
The hypothesis $\pdim_RN<\infty$ in the corollary is necessary, as otherwise
the sum in \eqref{eq:sum} is not defined.  
On the other hand, when $R$ is regular---and so $\pdim_RN$ is necessarily finite---the 
conclusion of the corollary may~fail.  

For example, if $M$ is a finite free $R$-module of rank $r$ and $d=\dim R$, then 
\[
\bass n{\fm M} = 
  \begin{cases}
r\binom{d}{n-1} & \quad\text{for } n\ne d+1\,,
  \\
0 & \quad\text{for } n= d+1\,.
  \end{cases}
 \]

Indeed, this follows from the cohomology exact sequence induced by the 
exact sequence $0\to\fm M\to M\to M/\fm M\to 0$ because $\Ext nRkM=0$,
for $n\ne d$, the map $\varepsilon^{d}_M$ is bijective by the proof of (i)$\implies$(iv) 
in Theorem \ref{thm:regular}, and $\bass {i}{k}=\binom di$.
\end{remark}

\begin{chunk}
\label{ch:formal}
Bass numbers are often described in terms of the generating formal power series 
$\rI M=\sum_{n\in\BZ}\bass nM t^n$.   We also use the series 
$\rP M=\sum_{n\in\BZ}\beta^R_n(M) t^n$.  

In these terms, the formulas \eqref{eq:sum} for all $n\in\BZ$ can be restated as an equality
  \begin{equation}
    \label{eq:cor}
\rI M=\rI R P^{R}_{N}(t^{-1}) + st\rP k\,.
  \end{equation}
  \end{chunk}

\begin{chunk}
\label{ch:fiber}
Let $(S,\fm_{S},k)$ and $(T,\fm_{T},k)$ be local rings and let $\varepsilon_S\col S\to k\gets T :\varepsilon_T$
denote the canonical maps.  The \emph{fiber product} of $S$ and $T$ over $k$ is defined by the formula
  \[
S\times_kT:=\{(s,t)\in S\times T\mid \varepsilon_S(s)=\varepsilon_T(t)\}\,.
  \]
It is well known, and easy to see, that this is a subring of $S\times T$, which is local with maximal ideal 
$\fm=\fm_{S}\oplus\fm_{T}$ and residue field $k$.  Set $R=S\times_kT$.

Let $N$ and $P$ be finite modules over $S$ and $T$, respectively.  The canonical maps $S\gets R\to T$ 
turn $N$ and $P$ into $R$-modules, and for them Lescot \cite[2.4]{Le} proved
  \begin{equation}
    \label{eq:frac}
\frac{\rI{N}}{\rP k}= \frac{\rI[S]{N}}{\rP[S]k}
  \quad\text{and}\quad
\frac{\rI{P}}{\rP k}= \frac{\rI[T]{P}}{\rP[T]k}\,.
  \end{equation}

When $N/\fn_SN=V=P/\fm_TP$ holds for some $k$-module $V$ the fiber product
  \[
N\times_VP:=\{(n,p)\in N\times P\mid \pi^N(n)=\pi^P(p)\}
  \]
has a natural structure of finite $R$-module. 
  \end{chunk}

\begin{corollary}
  \label{cor:fiber}
With notation as in  \emph{\ref{ch:fiber}}, set $v=\rank_kV$ and $M=N\times_VP$.

If $S$ and $T$ are singular and $\pdim_{S}N$ and $\pdim_{T}P$ are finite, then
\[
\frac{\rI{\fm_{R}M}}{\rP k} = \frac{\rI[S]S P^{S}_{N}(t^{-1})}{\rP[S]k} + \frac{\rI[T]T P^{T}_{P}(t^{-1})}{\rP[T]k} + 2vt\,.
\]
\end{corollary}

\begin{proof}
We have $\fm M\cong \fm_{S}N \oplus \fm_{T}P$ as $R$-modules, whence the first equality below:
\begin{align*}
\frac{\rI{\fm M}}{\rP k} 
	&=  \frac{\rI{\fm_{S}N}}{\rP k} + \frac{\rI{\fm_{T}P}}{\rP k} \\
        &=   \frac{\rI[S]S P^{S}_{N}(t^{-1})}{\rP[S]k}+ vt + \frac{\rI[T]T P^{T}_{P}(t^{-1})}{\rP[T]k} + vt\,.
     \end{align*}
The second one comes by applying formulas \eqref{eq:frac} and \eqref{eq:cor}, in this order.
  \end{proof}

  \begin{notes}
For $N=R$ and $M=\fm$ Corollaries \ref{cor:max} and \ref{cor:fiber} specialize to Lescot's results \cite[1.8(2)]{Le} 
and \cite[3.2(1)]{Le}, respectively.  The proof presented above for Corollary \ref{cor:fiber} faithfully transposes
his derivation of \cite[3.2(1)]{Le} from \cite[1.8(2)]{Le}.

When $N$ is any finite $R$-module with $\fm N\ne0$ and $M$ is a submodule containing $\fm N$, it is proved 
in \cite[Theorem 4]{ext} that the Bass numbers of $M$ and $\fm N$ \emph{asymptotically} have the same size, 
measured on appropriate polynomial or exponential scales.  The \emph{closed formula} in Corollary \ref{cor:max} 
is a much more precise statement, but as noted in Remark \ref{rem:regular} it does not hold when $\pdim_RN$ 
is infinite or when $R$ is regular.
  \end{notes}

\end{document}